\documentclass[11pt,a4paper]{article}
\usepackage{amsmath}
\usepackage{amsfonts}
\usepackage{amssymb}
\usepackage{amsthm}
\usepackage{amscd}
\usepackage{graphicx}
\theoremstyle{plain}
\newtheorem{lemma}{Lemma}[section]
\newtheorem{prop}{Proposition}[section]
\newtheorem{theorem}{Theorem}[section]
\theoremstyle{definition}

\newtheorem{ex}{Example}[section]

\newcommand{\metric}[2]{\ensuremath{\langle #1, #2\rangle}}  
\newcommand{\nks}{\ensuremath{S^3\times S^3}}   
\newcommand{\lcc}{\ensuremath{\tilde \nabla}}   
\renewcommand{\epsilon}{\varepsilon}            

\title{On almost complex surfaces in the nearly~K\"ahler~$\nks$}
\author{
	\emph{John Bolton}\\
	Durham University,
    Dept of Mathematical Sciences, Science Laboratories\\
    South Rd.,
    Durham DH1 3LE, United Kingdom\\ 
    \emph{john.bolton@durham.ac.uk}
	\and
	\emph{Bart Dioos}\\
	KU Leuven, 
    Departement Wiskunde\\ Celestijnenlaan 200B,
    3001 Leuven, Belgium\\
	\emph{bart.dioos@wis.kuleuven.be}
	\and
	\emph{Luc Vrancken}\\
	LAMAV, Universit\'e de Valenciennes\\
	Campus du Mont Houy,
	59313 Valenciennes Cedex 9, France\\
	KU Leuven, 
    Departement Wiskunde\\ Celestijnenlaan 200B,
    3001 Leuven, Belgium\\
	\emph{luc.vrancken@univ-valenciennes.fr}
       }
\date{}

\begin{document}

\maketitle

\begin{center}
\emph{
To Franki Dillen, our colleague, advisor, friend
}
\end{center}

\begin{abstract}
We study almost complex surfaces in the nearly K\"ahler~$\nks$. We show that there is a local correspondence between almost complex surfaces and solutions of the $H$-surface equation introduced by Wente \cite{bolton:wente}.  We find a global holomorphic differential on every almost complex  surface, and show that when this differential vanishes, then the corresponding solution of the $H$-surface equation gives a constant mean curvature surface in  $\mathbb R^3$. We use this, together with a theorem of Hopf, to classify all almost complex 2-spheres. In fact there is essentially only one, and it is totally geodesic. More details, as well as the proofs of the various theorems are given in~\cite{bolton:bddv}.

Finally, we state two theorems, one of which states that locally  there are just two almost complex surfaces with parallel second fundamental form. 
\end{abstract}

\section*{Introduction}

This is a report on joint work of the authors together with Franki~Dillen about almost complex surfaces in the
nearly K\"ahler~$\nks$. This was one of the last research projects in which Franki participated.

Nearly K\"ahler manifolds are almost Hermitian manifolds with almost complex structure~$J$ for which the tensor field~$\nabla J$
is skew-symmetric. In particular, the almost complex structure is non-integrable if the manifold
is non-K\"ahler.  Recently it has been shown
by Butruille~\cite{bolton:butruille} that there are only four homogeneous 6-dimensional non-K\"ahler, nearly K\"ahler manifolds, namely
the nearly K\"ahler 6-sphere $S^6$, the product $\nks$ of two 3-spheres, the projective space~$\mathbb{C}P^3$ and  $SU(3)/U(1)\times U(1)$, the manifold of flags in~$\mathbb{C}^3$.
All these spaces are compact 3-symmetric spaces.

It is natural to study two types of submanifolds of nearly K\"ahler (or more generally, almost Hermitian) manifolds,
namely almost complex and totally real submanifolds. Almost complex submanifolds are submanifolds
whose tangent spaces are invariant under $J$. Six-dimensional non-K\"ahler nearly K\"ahler manifolds do not admit 4-dimensional almost complex
submanifolds (\cite{bolton:podesta}), and almost complex surfaces 
in the nearly K\"ahler 6-sphere $S^6$ have been studied by many authors, see~e.g.\ \cite{bolton:boltonpeditwoodward}, \cite{bolton:bolton}, \cite{bolton:bryant}, \cite{bolton:dillenalmostcomplex}, \cite{bolton:grayalmostcomplex} and~\cite{bolton:sekigawa}. 
Some results have also been obtained \cite{bolton:xufeng} for almost complex surfaces in the  nearly K\"ahler $\mathbb{C}P^3$.

In the current article we  present a summary  of \cite{bolton:bddv}, in which we initiate the study of almost complex surfaces of the nearly K\"ahler~$\nks$. We show that there is a local correspondence between almost complex surfaces in the  nearly K\"ahler~$\nks$ and solutions of the $H$-surface equation introduced by Wente \cite{bolton:wente}.  We also find a global holomorphic differential on every almost complex surface in the nearly K\"ahler~$\nks$, and show that when this differential vanishes, then the corresponding solution of the $H$-surface equation gives a constant mean curvature surface in  $\mathbb R^3$.

In this note we use the fact that all holomorphic differentials on a topological 2-sphere must be identically zero and a well-known theorem of Hopf on constant mean curvature 2-spheres  to show that up to nearly K\"ahler isometries of $\nks$ there is just one almost complex topological 2-sphere in~$\nks$. This 2-sphere is totally geodesic and has constant curvature 2/3. This is rather different from the case of the nearly K\"ahler 6-sphere; there are many almost-complex 2-spheres in the nearly  K\"ahler 6-sphere, even one of constant curvature,  which are not totally geodesic (see \cite[\S\,5, Example 2]{bolton:sekigawa}).

We conclude this note by stating two theorems, one of which says that, locally at least, there are just two almost complex surfaces in $\nks$ with parallel second fundamental form, namely a flat torus and a constant curvature 2-sphere, both of which are totally geodesic. These theorems are proved  in \cite{bolton:bddv} using techniques similar to those outlined in this note.

\section{The nearly K\"ahler $\nks$} \label{bolton:nks3s3}

In this section, we describe a  nearly K\"ahler structure on $S^3 \times S^3$. 
Care is needed here, as the metric involved is not the product metric.

Let $S^3$ be the unit 3-sphere regarded as the group of unit quaternions, so
\[
S^3=\{x+yi+zj+wk  \mid x^2+y^2+z^2+w^2=1\}.
\]
For $p,q \in S^3$, we identify $T_{(p,q)}(S^3\times S^3)$ with $T_pS^3 \times T_qS^3$ in the usual manner and begin our definition of an almost complex structure $J$ on $\nks$ by defining  $J$ on tangent vectors $(\alpha,\beta)\in T_{(1,1)}(S^3\times S^3)$  by
\[J(\alpha,\beta)=\frac{1}{\sqrt 3}(2\beta-\alpha, -2\alpha+\beta).
\]
Then, if $(U,V)\in T_{(p,q)}(S^3\times S^3)$, we use quaternion multiplication to translate back to $(1,1)$, i.e.
\[
(U,V)\mapsto (p^{-1}U, q^{-1}V),
\]
then apply $J$ as defined above to give
\[
(U,V)\mapsto \frac{1}{\sqrt 3}(2q^{-1}V-p^{-1}U, -2p^{-1}U+q^{-1}V)
\]
and then translate back to $(p,q)$ to give \cite{bolton:butruille}
\begin{equation}\label{bolton:acstr}
J(U,V)= \frac{1}{\sqrt 3}(2pq^{-1}V-U, -2qp^{-1}U+V).
\end{equation}
An easy check shows that $J^2=-\mathrm{Id}$.

 The standard product metric $\metric{\cdot}{\cdot}$ on $S^3\times S^3$ is not $J$-invariant, but we may use it to define a  $J$-invariant metric $g$ on $S^3 \times S^3$ in a natural way by taking
\begin{equation}\label{bolton:metric}
g(X,Y)=\frac{1}{2}(\metric{X}{Y}+\metric{JX}{JY}), \quad X,Y \in T_{(p,q)}(S^3\times S^3).
\end{equation}

One can work out the Riemannian connection $\tilde\nabla$ for $g$, and it 
turns out that $(S^3\times S^3, g, J)$ is a {\em nearly K\"ahler} manifold meaning that 
\[
J^2=-\hbox{Id}, \quad g(JX,JY)=g(X,Y), \quad (\lcc_X J )X=0
\]
for  $X,Y \in T_{(p,q)}(\nks)$.

In fact  $(S^3\times S^3, g, J)$ is a homogeneous nearly K\"ahler manifold with nearly K\"ahler isometries given by
\[
(p,q)\mapsto (apc^{-1},bqc^{-1}),  
\]
$a,b,c$ being unit quaternions.

\section{The curvature tensor}

 In this section we write down the curvature tensor of the nearly K\"ahler $S^3\times S^3$. In order to do this it is convenient to define a new tensor $P$. Proceeding as we did with $J$, we first define $P$ at $(1,1)$ and then use quaternion multiplication to move it round the whole space. 

So take
\[
P(\alpha,\beta)=(\beta,\alpha), \quad (\alpha,\beta)\in T_{(1,1)}(S^3\times S^3),
\]
and then, if $(U,V)\in T_{(p,q)}(\nks)$, we define 
\begin{equation} \label{bolton:pdef}
P(U,V)=(pq^{-1}V, qp^{-1}U).
\end{equation}

 We call $P$ an {\em almost product structure} because it reflects the product structure, but is not parallel. Easy checks show that
\[
P^2=\hbox{Id}, \quad PJ=-JP, \quad g(PX,PY)=g(X,Y)
\]
for~$X$, $Y$ tangential to~$\nks$.

It then turns out that the curvature tensor $\tilde R$ is given by
\begin{equation*}
 \begin{split}
  \tilde R(U,V)W &= \frac{5}{12}\bigl(g(V,W)U - g(U,W)V\bigr) \\
                 &\quad  +\frac{1}{12}\bigl(g(JV,W)JU - g(JU,W)JV - 2g(JU,V)JW\bigr) \\
                 &\quad + \frac{1}{3}\bigl(g(PV,W)PU - g(PU,W)PV  \bigr.\\
                        &\quad  \phantom{\frac{1}{3}}\quad\mbox{ } + \bigl. g(PJV,W)PJU - g(PJU,W)PJV\bigr),
 \end{split}
\end{equation*}
A straightforward calculation using the above expression for $\tilde R$ shows that

\begin{lemma} \label{bolton:lemma1}Let $\Omega$ be a $J$-invariant 2-plane. If $P(\Omega)$ is  perpendicular to $\Omega$, then $\Omega$ has sectional curvature $\tilde K$ equal to $2/3$. On the other hand, if $P(\Omega)=\Omega$, then $\Omega$ has $\tilde K= 0$.
\end{lemma}

\section{Almost complex surfaces}

{\bf Definition}\quad A smooth surface $M$ in a nearly K\"ahler manifold is said to be an {\em almost complex} surface if the tangent bundle $TM$ of $M$ is $J$-invariant. 

Standard arguments using the Gauss equation show:

\begin{lemma} \label{bolton:lemma2}An almost complex surface in a nearly K\"ahler manifold  is minimal, and is totally geodesic if and only if $\tilde K= K$, where $\tilde K$ is the sectional curvature of the tangent plane as a plane in the nearly K\"ahler manifold and $K$ is the sectional curvature of the induced metric.
\end{lemma}

 This lemma is useful because it does not seem straightforward to compute the second fundamental form directly.

  We first look for almost complex surfaces in $S^3\times S^3$ which are also totally geodesic. The almost product structure $P$ plays a large role.

\begin{prop} \label{bolton:totgeo}If an almost complex surface $M$ in $S^3\times S^3$ is totally geodesic, then either
\begin{enumerate}
\item[(i)] $P(TM)\perp  TM$, in which case  $K=\tilde K= 2/3$, or
\item[(ii)] $P(TM) =  TM$, in which case $K=\tilde K=0$.
\end{enumerate}
\end{prop}

\begin{proof}
If $X$ is a unit tangent vector to a totally geodesic surface, then $\tilde R(X,JX)X$ is also tangential to the surface.
Therefore, since $M$ is an almost complex surface, it  is a scalar multiple of $JX$.  

The fact that either $P(TM)\perp  TM$ or $P(TM) =  TM$ may now be proved from the expression for $\tilde R$  obtained in the previous section by considering a unit tangent vector $X$ at $p\in M$ for which $g(PX,X)$ is maximal for all unit vectors tangential to $M$ at $p$.

Lemmas \ref{bolton:lemma1} and \ref{bolton:lemma2} then show  $\tilde K=K =\frac{2}{3}$ in the former case and  $\tilde K= K=0$ in the latter case.
\end{proof}

\section{Two examples}

We now give two examples of almost complex surfaces in $S^3\times S^3$ which are totally geodesic, one to illustrate each of the possibilities given in Proposition~\ref{bolton:totgeo}.

\begin{ex} \label{bolton:ex1} Let $\phi\colon\mathbb R^2 \to S^3\times S^3$ be given by
\[
\phi(s,t) = (\cos s+i\sin s, \cos t+i\sin t).
\]
 A short calculation shows that the image $M$ of this immersion is almost complex and $P(TM) = TM$. 
 It is also easy to check that $g(\phi_s, \phi_s)=g(\phi_t, \phi_t)=4/3$, and  $g(\phi_s, \phi_t)=-2/3$. In particular, all are constant so that the induced metric has sectional curvature $K=0$. 
 That $\phi$ is totally geodesic now follows from Lemma \ref{bolton:lemma1} and Lemma \ref{bolton:lemma2}.
This gives a flat and totally geodesic almost complex torus in~$S^3\times S^3$.
\end{ex}

\begin{ex} \label{bolton:ex2} Let $S^2$ be the 2-sphere of unit imaginary quaternions, and let $\psi\colon S^2\to  S^3\times S^3$ be given by
\[
\psi(x)=\frac{1}{2}(1-\sqrt 3 x,1+\sqrt 3 x).
\]
 Calculations similar to those needed in the previous example show that the image $M$ of this immersion is an almost complex surface with  $P(TM)\perp  TM$. It is not hard to show that the induced metric is 3/2 times the standard metric on $S^2$, 
so that the induced sectional curvature~$K$ is 2/3. It now follows from Lemma~\ref{bolton:lemma1}  and Lemma~\ref{bolton:lemma2} that this almost complex 2-sphere is totally geodesic.
\end{ex}

\section{A holomorphic differential}

We now explore the mathematics of an almost complex surface $M$ in $\nks$ using isothermal coordinates $(u,v)$. 

 So let $\phi(u,v)=(p(u,v),q(u,v))\in S^3\times S^3$ be an almost complex surface in~$\nks$ with $J(p_u,q_u)=(p_v,q_v)$.  We may then write
\begin{equation}\label{bolton:abgd}
p^{-1}p_u=\alpha, \quad p^{-1}p_v=\beta, \quad q^{-1}q_u=\gamma, \quad q^{-1}q_v=\delta,
\end{equation}
where $\alpha$, $\beta$, $\gamma$, $\delta$ are tangent vectors to the set $S^3$ of unit quaternions at $1$. That is to say,  $\alpha$, $\beta$, $\gamma$, $\delta$ take values in the imaginary quaternions, which we will identify with $\mathbb R^3$ in the usual manner.

 Then $(p_u,q_u)=(p\alpha,q\gamma)$ and $(p_v,q_v)=(p\beta, q\delta)$, so the almost complex condition  $J(p_u,q_u)=(p_v,q_v)$ enables us to find $\gamma$ and~$\delta$ in terms of  $\alpha$ and~$\beta$. Specifically,
\begin{equation}\label{bolton:gd}
\gamma = \frac{\sqrt 3}{2}\beta +\frac{1}{2}\alpha, \quad \delta = \frac{1}{2}\beta - \frac{\sqrt 3}{2}\alpha.
\end{equation}
This then enables us to show, using \eqref{bolton:acstr} and \eqref{bolton:metric}, that the metric induced on the almost complex surface~$M$ is 
\begin{equation}\label{bolton:indmet}
(\alpha\cdot\alpha + \beta\cdot\beta)(du^2+dv^2),
\end{equation}
where ``$\,\cdot\,$'' denotes the standard inner product in $\mathbb R^3$.

 Using \eqref{bolton:abgd} the integrability condition $p_{uv}=p_{vu}$
 gives that
\begin{equation}\label{bolton:a-b}
\alpha_v-\beta_u=\alpha\beta-\beta\alpha=2\alpha \times \beta,
\end{equation}
where ``$\times$'' denotes the vector cross product on $\mathbb R^3$.

 The similar condition for $q$, after using \eqref{bolton:gd} to substitute for $\gamma$ and~$\delta$ in terms of $\alpha$ and~$\beta$, gives
\begin{equation}\label{bolton:a+b}
\alpha_u+\beta_v=\frac{2}{\sqrt 3} \alpha \times \beta.
\end{equation}

A short calculation now gives
\[
(\alpha\cdot\beta)_u=\frac{1}{2}(\alpha\cdot\alpha - \beta\cdot\beta)_v
\text{ \quad and \quad }
(\alpha\cdot\beta)_v=-\frac{1}{2}(\alpha\cdot\alpha - \beta\cdot\beta)_u
\]
which are the Cauchy-Riemann equations for the complex function $2\alpha\cdot\beta+i(\alpha\cdot\alpha - \beta\cdot\beta)$, which is thus complex differentiable. The following theorem may now be proved using some elementary algebra,  the definition of the metric $g$ in \eqref{bolton:metric} and $P$ given in \eqref{bolton:pdef}.

\begin{theorem} \label{bolton:L} Let $(u,v)$ be isothermal coordinates on an almost complex surface $M$ in $ S^3\times S^3$ and let $\Lambda= g(P\phi_z,\phi_z)$. Then $\Lambda dz^2$ is a globally defined holomorphic differential and the following three conditions are equivalent.
\begin{enumerate}
\item[(i)] $\Lambda dz^2 =0$,
\item[(ii)]  $\alpha\cdot\alpha - \beta\cdot\beta=0$ and $\alpha\cdot\beta=0$,
\item[(iii)] $P(TM) \perp TM$. 
\end{enumerate}
\end{theorem}

\section{Link with $H$-surface equation}

We now change direction slightly and look again at equations \eqref{bolton:a-b} and \eqref{bolton:a+b}. If we precede $\alpha$ and~$\beta$ by a rotation in the tangent spaces of $M$ through angle $2\pi/3$ to give $\tilde\alpha =-\frac{1}{2}\alpha+\frac{\sqrt 3}{2}\beta$ and $\tilde\beta=-\frac{\sqrt 3}{2}\alpha-\frac{1}{2}\beta$, then  equations \eqref{bolton:a-b} and \eqref{bolton:a+b} become
\begin{equation}\label{bolton:avbu}
\tilde\alpha_v=\tilde\beta_u, 
\end{equation}
and 
\begin{equation}\label{bolton:aubv}
\tilde\alpha_u+\tilde\beta_v=-\frac{4}{\sqrt 3} \tilde\alpha \times \tilde\beta.
\end{equation}

Equation \eqref{bolton:avbu} is an integrability condition. It shows that the form 
 $\tilde\alpha du+\tilde\beta dv$ is closed. Hence, if the surface $M$ is simply-connected, there exists an immersion $\epsilon:M \to \mathbb R^3$  with $\epsilon_u=\tilde\alpha$ and  $\epsilon_v=\tilde\beta$.  

Equation \eqref{bolton:aubv} now gives
$$
\epsilon_{uu}+\epsilon_{vv}=-\frac{4}{\sqrt 3} \epsilon_{u}\times\epsilon_{v},
$$
which is the $H$-surface equation of Wente \cite{bolton:wente}. In fact this process may be reversed: there is, locally at least,  a correspondence between almost complex surfaces in  $ S^3\times S^3$ and solutions of the  $H$-surface equation. It is clear that when $(u,v)$ are isothermal coordinates for $\epsilon$ the image is a surface  in $\mathbb R^3$ with constant mean curvature $H= -2/\sqrt 3$.

\begin{theorem} \label{bolton:accmc}Let $M$ be a simply connected almost complex surface in $ S^3\times S^3$ with  $\Lambda dz^2=0$. Then there exists a corresponding immersion $\epsilon:M \to \mathbb R^3$  with $\epsilon_u=\tilde \alpha$ and  $\epsilon_v=\tilde \beta$, and this immersion has constant mean curvature equal to $-2/\sqrt 3$.
Also the metric induced by $\epsilon$ is half that induced on the corresponding almost complex surface.
\end{theorem}

\begin{proof} 
 We need to prove that $(u,v)$ are isothermal coordinates for $\epsilon$. That is to say that $\tilde \alpha\cdot\tilde \alpha - \tilde \beta\cdot\tilde \beta=0$ and $\tilde  \alpha\cdot\tilde \beta=0$.
However, since $\tilde \alpha$, $\tilde \beta$ are obtained from $\alpha$, $\beta$ by a rotation, it follows that this holds if and only if $\alpha\cdot\alpha - \beta\cdot\beta=0$ and $\alpha\cdot\beta=0$. Theorem \ref{bolton:L} now shows that this holds if and only if  $\Lambda dz^2=0$

We now note that, in this case,  the metric induced by  $\epsilon$ is equal to $\tilde\alpha\cdot\tilde\alpha(du^2+dv^2)=\alpha\cdot\alpha(du^2+dv^2)$, so the final statement of the theorem follows from \eqref{bolton:indmet}.
\end{proof}

\section{Almost complex 2-spheres}

In this section, we outline a proof of the following theorem.

 \begin{theorem} Every almost complex 2-sphere $S^2$ in $S^3\times S^3$ is totally geodesic and is obtained by applying a nearly K\"ahler isometry of $\nks$ to the immersion given in Example \ref{bolton:ex2}.
\end{theorem}

\begin{proof} 
 Since $S^2$ admits no non-zero holomorphic differentials, for an almost complex 2-sphere  $\Lambda dz^2=0$. Hence the corresponding  solution $\epsilon$ of the $H$-surface equation has constant mean curvature. By Hopf's theorem, this must be the round sphere, and since it has mean curvature  $-2/\sqrt 3$ it has radius $\sqrt 3/2$ and hence constant sectional curvature~$4/3$. 

However, as noted in Theorem \ref{bolton:accmc}, the metric induced on $S^2$ by  $\epsilon$ is half that induced on the almost complex surface, so this latter has constant sectional curvature $K=2/3$. 

Theorem \ref{bolton:L} shows that  $P$ maps the tangent spaces of the surface to normal spaces, so, by Lemma \ref{bolton:lemma1}, the sectional curvature $\tilde K$ is also equal to $2/3$. 
It now follows from Lemma \ref{bolton:lemma2} that the almost complex surface is totally geodesic. 

We note that changing the almost complex surface by an isometry~$(p,q) \mapsto (apc^{-1},bqc^{-1})$,
where~$a,b,c$ are unit quaternions, corresponds to changing~$\alpha$ and~$\beta$ to~$c \alpha c^{-1}$
and~$c \beta c^{-1}$ respectively.
Since $S^3$ is the double cover of~$SO(3)$, $\alpha$ and $\beta$ change by a rotation in~$\mathbb{R}^3$, 
and thus $\epsilon$ changes by an isometry of $\mathbb{R}^3$.

The uniqueness part then follows essentially from this fact and the uniqueness
 of constant mean curvature 2-spheres in $\mathbb R^3$.
\end{proof}

In \cite{bolton:bddv}, we use similar techniques to classify other types of almost complex surface in $S^3\times S^3$.

\begin{theorem} Let $M$ be an almost complex surface in $S^3\times S^3$ such that  $P(TM) = TM$. Then $M$  may be  obtained by applying a nearly K\"ahler isometry of $\nks$ to the flat totally geodesic torus example given in Example \ref{bolton:ex1}.
\end{theorem}

\begin{theorem} If $M$ is an  almost complex surface in $S^3\times S^3$ with parallel second fundamental form, then $M$ is totally geodesic. Locally at least,  $M$  may be  obtained by applying a nearly K\"ahler isometry of $\nks$ to the immersion given in Example \ref{bolton:ex1} or  Example \ref{bolton:ex2}. 
\end{theorem}

The proofs of these theorems are a little harder than those we have discussed in this note, and involve the covariant derivatives (using the Levi-Civit\/a connection $\tilde\nabla$ of the nearly K\"ahler structure on $S^3\times S^3$) of the almost complex structure $J$ and of the almost product tensor $P$.

\nocite{*}
\bibliographystyle{amsplain}
\bibliography{nearlykahlerbib}

\end{document}